 \documentclass[12pt,cmcyralt]{article}
 \usepackage[utf8]{inputenc}
 \usepackage{cmap}
 \usepackage{graphicx}
 \usepackage[T2A]{fontenc}
 \usepackage[russian,english]{babel}
 \usepackage{amsmath,amssymb,euscript,amsthm,amsfonts,mathrsfs,amscd}
 \usepackage{color}
 \usepackage{floatflt}
 \usepackage{mathrsfs} 
 \usepackage[ left=20mm, top=20mm, right=20mm, bottom=20mm, footskip=10mm, nohead, nomarginpar, driver=xetex]{geometry}
 \theoremstyle{plain}
 \newtheorem{thm}{Theorem}
 \newtheorem{lem}{Lemma}
 
 \theoremstyle{Theorem}
 \newtheorem{definition}{Definition}
 
 \theoremstyle{remark}
 \newtheorem{rem}{Remark}

\newcommand{\bk}{\color{black}}
\newcommand{\bb}{\color{blue}}
\newcommand{\rr}{\color{red}}
\newcommand{\dd}{\ensuremath{\displaystyle}}

\newcommand\intl{\ensuremath{\int\limits}}
\newcommand{\suml}{\ensuremath{\sum\limits}}
\newcommand{\bD}{\stackrel{\mathscr{D}}{=}}
\newcommand{\1}{\ensuremath{\mathbf{1}}}

\newcommand{\ud}{\,\mathrm{d}}

\newcommand{\bd}{\stackrel{\text{\rm def}}{=\!\!\!=}}
\newcommand{\PP}{\ensuremath{\mathsf{P}}}

\newcommand{\EE}{\ensuremath{{\mathbb{E}}}}

\newcommand*{\QEDB}{\hfill \ensuremath{\square}}%
\newcommand*{\DIAM}{\hfill \ensuremath{\Diamond}}
\newcommand*{\TR}{\hfill \ensuremath{\triangleright}}

 \author{Elmira ~Yu.~Kalimulina\footnote{Institute of Control Sciences of   Russian Academy of Sciences, Russian University of Transport} \label{note1}\\
Galina~ A.~ Zverkina \footnote{Institute of Control Sciences of   Russian Academy of Sciences, Russian University of Transport}}

\title{On some generalization of Lorden's inequality for renewal processes\thanks{
The authors are supported by the Russian Foundation for Basic Research project no. 20-01-00575\_a}}

\begin{document}
 \maketitle
 \sloppy

\section{Introduction}
Let's consider a renewal process $N_t\bd\dd\suml_{i=1}^\infty \1\left\{ \suml_{k=1}^i \xi_k\leqslant t\right\}  $, where
$\left\{\xi_1, \xi_2, ...\right\} $ are independent identically distributed (i.i.d.) positive random variables.
$N_t$ is a counting process with jumps, 
$t_k\bd\dd\suml_{k=1}^i \xi_k$ is referred to as the $k$-th jump time.  
The times $t_k$ are renewal moments of $N_t$.

Consider the {\it backward renewal time (or overshoot) at some time $t$} (See Fig.\ref{fig1:backward renewal time}):
$$B_t=t-\sum_{k=1}^{N_t} \xi _k.$$

$B_t$ is called a backward renewal time at the fixed time. So we can consider 
 $B_t$ as a random process at arbitrary time $t$. It's easy to show $B_t$ is a Markov process.

~\\

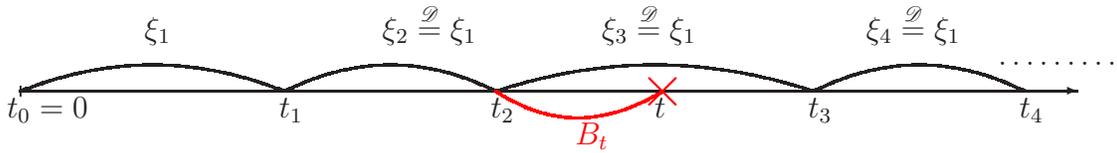
\begin{figure}[h]\label{fig1:backward renewal time}
    \centering
   \begin{picture}(400,45)
\put(0,20){\vector(1,0){400}}
\put(0,18){\line(0,1){4}}
\thicklines
{\qbezier(0,20)(50,40)(100,20)}
{\qbezier(100,20)(140,40)(180,20)}
\qbezier(180,20)(240,40)(300,20)
\qbezier(300,20)(340,40)(380,20)
\put(370,30){\ldots \ldots \ldots}
\put(98,10){$t_1$}
\put(178,10){$t_2$}
\put(298,10){$t_3$}
\put(378,10){$t_4$}
\put(47,40){$\xi_1$}
\put(137,40){$\xi_2\bD \xi_1$}
\put(220,40){$\xi_3\bD \xi_1$}
\put(320,40){$\xi_4\bD \xi_1$}
\put(-5,10){$t_0=0$}
\put(240,10){$t$}
\rr
\put(238,15){\line(1,1){10}}
\put(238,25){\line(1,-1){10}}
\qbezier(180,20)(210,0)(243,20)
\put(210,0){$B_t$}
%
\bk
\end{picture}
    \caption{{\bb$B_t$} is a backward renewal time at the fixed time $t$.}
    \label{fig1}
\end{figure}

\begin{thm}[Lorden, G. (1970) \cite{Lorden}; see, e.g. \cite{Chang}]\label{thm} 
Lorden's inequality states that the expectation of this overshoot is bounded as
\begin{equation}\label{lord}
\EE\, B_t \leqslant \frac{\EE\, \xi^2}{\EE\, \xi}.
 \end{equation}
 \QEDB
\end{thm}

Consider here also {\it forward renewal time (or undershot ) at some time $t$} (See Fig.\ref{fig1:forward renewal time}):
$$W_t=t-\sum_{k=1}^{N_t+1} \xi _k.$$
~\\

\begin{figure}[h]\label{fig1:forward renewal time}
    \centering
   \begin{picture}(400,45)
\put(0,20){\vector(1,0){400}}
\put(0,18){\line(0,1){4}}
\thicklines
{\qbezier(0,20)(50,40)(100,20)}
{\qbezier(100,20)(140,40)(180,20)}
\qbezier(180,20)(240,40)(300,20)
\qbezier(300,20)(340,40)(380,20)
\put(370,30){\ldots \ldots \ldots}
\put(98,10){$t_1$}
\put(178,10){$t_2$}
\put(298,10){$t_3$}
\put(378,10){$t_4$}
\put(47,40){$\xi_1$}
\put(137,40){$\xi_2\bD \xi_1$}
\put(220,40){$\xi_3\bD \xi_1$}
\put(320,40){$\xi_4\bD \xi_1$}
\put(-5,10){$t_0=0$}
\put(240,10){$t$}
\rr
\put(238,15){\line(1,1){10}}
\put(238,25){\line(1,-1){10}}
\qbezier(243,20)(271,0)(300,20)
\put(270,0){$W_t$}
%
\bk
\end{picture}
    \caption{{\rr$ W_t$} is a forward renewal time at the fixed time $t$.}
    \label{fig1}
\end{figure}
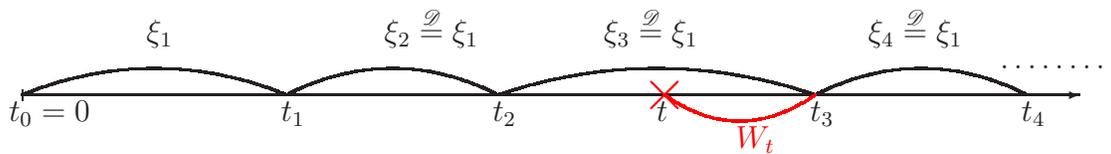

The renewal intervals can be dependent, and also may be different distributed.
In this paper the random variables $\left\{\xi^\ast_1, \xi^\ast_2, ...\right\} $
are non-negative and not assumed to be i.i.d.
The counting  process $N^\ast_t\bd \dd \suml_{i=1}^\infty \1\left\{\suml_{k=1}^i \xi^\ast_k\leqslant t\right\} $ is called as {\it generalised renewal process} in this case.
Our goal is to 
generalise the Lorden's inequality for that process, and to formulate the conditions under which this  {\it generalised } Lorden's inequality is hold.

\section{Assumptions}
First, we have to define the {\it generalised intensity function}. 
Recall the intensity (hazard rate) function definition (in a classical sense \cite{GK,Smith}). Let $\PP\{\cdot \}$ to be the probability of at least one recovery occurs  in the interval $[t,~t+\Delta]$. It can be expressed via some function $\varphi(t)$:
\begin{eqnarray*}
\PP\{\mbox{at least one recovery  in the interval }[t, t+\Delta]\} 
\\
= \frac{F(t+\Delta)- F(t)}{1-F(t)}
=\intl_t^{t+\Delta}\ \varphi(s)\ud s=\varphi(t)\Delta + o(\Delta), 
\end{eqnarray*}
\begin{definition}
$\varphi(t)$ is called the  intensity (hazard rate) function for a renewal process.\TR
\end{definition}

It is obvious  that  a continuous  random variable  is uniquely defined by its cumulative distribution functions, or by its density distribution functions, or by its intensity:
$$
F(s)=1-\exp\left(\intl_0^s(-\varphi(u))\ud u\right), \qquad F'(s)=\varphi(s)\exp\left(\intl_0^s(-\varphi(u))\ud u\right).
$$
Functions $F(s), F'(s), \varphi(s)$ uniquely define each other.

This above definition of intensity is formulated for absolutely continuous distributions. We 
will consider the more generalised case of mixed random variables\footnote{The case of singular random variables  is not considered due there are no practical applications for that case.} with distribution functions that may have a countable number of  jumps.

{\bf Denote} $\varphi(a)\bd-\ln \big(F(a+0)-F(a-0) \big)\delta(0)$
when  $F(a-0)\neq F(a+0)$. The function  $\delta(\cdot)$ is the Dirac delta function.

{\bf Suppose that}
$$
f(s)=
\begin{cases}
F'(s), & \mbox{if } F'(s) \mbox{exists};\\
0, & \mbox{in the other case}.
\end{cases}
$$
{\it This is not ``classical'' distribution density!!!}

\begin{definition} The generalized intensity is defined by: 
$$\varphi(s)\bd \dd\frac{f(s)}{1-F(s)}-\suml_{i}\delta(s-a_i)\ln\big(F(a_i+0)-F(a_i-0) \big),$$ where  $\{a_i\}$ --- is the set of all points of discontinuity of a function 
$F(s)$.\TR
\end{definition}

{\bf Denote} the generalized intensity of a random variable  $\xi$ as
$$
Intensity_\xi(x). \eqno{\TR}
$$

~

~
The following lemma for generalised intensity is hold

\begin{lem} If $\xi\perp\!\!\!\perp\eta$, then
$$
Intensity_{\min\{\xi;\eta}(x)=Intensity_\xi(x)+Intensity_\eta(x).\eqno{\TR}
$$
\end{lem}

Further we will formulate our results for generalized intensity. All three cases (cumulative probability density probability and generalized intensity functions) will be used for a random variable definition.   

~

~

Consider the sequence $\left\{\xi_1, \xi_2, ...\right\} $ of random variables.

Let's make the following assumptions (1)--(5):
\begin{enumerate}
\item $\xi_j=\min\{\zeta_j;\theta_j\}$, where  $\{\zeta_j\}$ -- i.i.d. r.v., defined by the generilized intensity $\varphi_i(s)$, and $\zeta_i \perp\!\!\!\perp \theta_j$ for all  $i$, $j$; $\theta_j$ is defined by generalized intensity   $\mu_j$;\footnote{The random variables are supposed to be non-identically distributed. But the formulated results still hold  for the condition $\xi_i \bD \xi_j$  for all $i,~j$, $i\neq j$, or not for all $i,~j$.} 

\item  The generalized measurable non-negative function $Q(s)$ exists, \\
and for all  $s\geqslant 0$~~$\boxed{\varphi(s)+\mu_j(s)=\lambda_i(s) \leqslant  Q(s)}$;

\item $\dd\intl_0^\infty \varphi(s) \ud s = \infty$, and $\dd\intl_0^\infty \left( x^{k-1}  \exp\left(-\intl_0^x \varphi(s)\ud s\right)\right)\ud x<\infty $ for some $~k\geqslant 2$;

\item
$Q(s)$ is locally bounded function for some neighbourhood of zero;

\item $\varphi(s)>0$ a. e. if $s>T \geqslant 0$.

\end{enumerate}

\begin{definition}\label{def1}
If conditions 1--4 are satisfied, then the counting process
\begin{equation}\label{def} 
  N_t\bd\suml_{i=1}^\infty \1\left\{\suml_{k=1}^i \xi_k\leqslant t\right\}
\end{equation}
is called a {\it generalized renewal process}.
\DIAM
\end{definition}

\begin{rem}
If $T>0$, the process \eqref{def} is the delayed process.\TR
\end{rem}

\begin{rem}\label{r0}
The condition (3) ensures that:
 $$\EE\,\xi_i>0,\qquad \mbox{Var}\,\xi_i^2>0.$$
 \TR
\end{rem}

\begin{rem}\label{r1}
 If the condition 4 holds, then:
 $$F_i(t)=1-\dd\intl_0^t \exp({-\varphi_i(s)})\ud s \geqslant 1- \frac{1}{(1+t)^c}\;\; \Rightarrow \;\;\exists\; \EE\,\xi_i^2 < \infty.$$
 \TR
\end{rem}

\begin{rem}
The mixed  random variable  is uniquely defined by its cumulative distribution functions, or by its intensity.
These functions $F(s), \varphi(s)$ uniquely define each other. \TR
\end{rem}

\section{Auxiliary results}

Let's consider random variables $\zeta$, $\xi_i$, $\eta$ with the following intensities and distribution functions:
\begin{eqnarray}
\bullet \quad G(x)\bd\PP\{\zeta\leqslant x\}=1-\exp\left(-\dd\intl_0^x Q(s)\ud s\right);  \; Q(s)  \mbox { is an intensity};\label{G}
\\
\bullet \quad F_i(x)\bd  \PP\{\xi_i\leqslant x\}=1-\exp\left(-\dd\intl_0^x \lambda_i(s)\ud s\right);  \mbox { where  $\lambda_i(s)$ -- intensity};\label{Fi}
\\
\bullet \quad \Phi(x)\bd\PP\{\eta\leqslant x\}=1-\exp\left(-\dd\intl_0^x \varphi(s)\ud s\right); \; \varphi(s) \mbox {  is an intensity}.\label{Phi}
\end{eqnarray}
The condition 1 ensures that $G(s)=\PP\{\zeta\leqslant s\}\geqslant  F_i(s)=\PP\{\xi_i\leqslant s\} \geqslant  \Phi(s) = \PP\{\eta\leqslant s\}$, or $\zeta\prec \xi_i \prec \eta$ -- ordered by distribution \cite{SHT}.

The condition 3 ensures that $\EE\,\eta^k<\infty \Rightarrow \EE\,\zeta^k<\infty$ and $\EE\,\xi_i^k<\infty$.

The condition 4 ensures that $\EE\,\zeta>0$.

The condition 5 ensures that $\Phi'(x)>0$ a.e. if $s>T$.

~

The condition $\zeta\prec \xi_i \prec \eta$ gives some useful auxiliary result.

\begin{lem}
The following inequalities hold for the generalized renewal process:
 $$G(s)^{\ast n}\geqslant F_{\xi_1+ \xi_2+ \cdots+\xi_n}(s)\geqslant \Phi^{\ast n}(s),$$
 or $$\suml_{i=1}^n\zeta_i \prec \suml_{i=1}^n \xi_i  \prec \suml_{i=1}^n\eta_i. \eqno{\TR}$$
\end{lem}

\section{The main result}
Let us consider the  counting process \eqref{def}, where $\xi_j$ -- r.v., that may be dependent.

Let $\PP\{\xi_j\leqslant s\}=F_j(s)$;   $F_j$ and $F_i$ may not be equal.

\begin{thm}\label{thm1} 
If the conditions 1--5 are satisfied, then the following inequality for the process defined by the Definition  \ref{def1}
 holds: 
\begin{equation}\label{osn}
\EE\,B_t\leqslant  \EE\,  \eta +  \frac{\EE\,\eta^2}{2\EE\,\zeta}
\end{equation}
 and 
 \begin{equation}\label{osn1}
  \EE\,W_t\leqslant  \EE\,  \eta +  \frac{\EE\,\eta^2}{2\EE\,\zeta},
\end{equation}
where
$ \qquad\dd \EE\,  \eta^2= \intl_0^\infty x^2 \ud \Phi(x);\qquad
 \EE\,  \zeta= \intl_0^\infty x^2 \ud G(x)$ -- see (\ref{G}), (\ref{Phi}). \QEDB 
\end{thm}
\begin{proof}

In many books on the renewal theory, there is the standard estimation for the distribution of the backward renewal time:
\begin{multline*}
\PP(B_t > x)= 1-F_1(t+x)+\sum_i^\infty \intl_0^{x-t} 1-F_{i+1}(x-s) \ud F_{\xi_1+\xi_2+\cdots+\xi_i}(s)\leqslant 
\\
\leqslant  1-\Phi(t+s)- \sum_{i=1}^\infty \intl_0^{x-t}1- \Phi(x-s) \ud F_{\xi_1+\xi_2+\cdots+\xi_i}(s)=
\\
=1-\Phi(t+s)- \sum_{i=1}^\infty I_i,
\end{multline*}
where $\dd \Phi(s)$ is defined in (\ref{Phi}),
 and  $F_{\xi_1+\xi_2+\cdots+\xi_i}(s)=\PP\left\{\xi_1+\xi_2+\cdots+\xi_i\leqslant s\right\} .$

Here, we denote:
\begin{enumerate}
    \item $\zeta$ is a random variable with the distribution function $ G(x)=1-\dd\intl_0^x \exp ({-Q(s)})\ud s$;
    \item $\eta$ is a random variable with the distribution function $\Phi(x)=1-\exp\left(-\dd\intl_0^x \varphi(s)\ud s\right)$.
\end{enumerate}

So, $G(s)\geqslant F_i(s)\geqslant \Phi(s);$ and $\zeta \prec  \xi_i  \prec \eta$ (distribution order).

Note, that
$$
G(s)\geqslant F_i(s)\geqslant \Phi(s);\quad \zeta \prec  \xi_i  \prec \eta\quad (\mbox{ distribution order }), $$
 then
\begin{equation}\label{sravn}
    G^{*n}(x) \geqslant F_{\xi_1+\xi_2+\cdots+\xi_i}(x) \geqslant \Phi^{*n}(x).
\end{equation}

Now, let us estimate
$$
I_i= \intl_0^{x-t} 1-F_{i+1}(x-s)\ud F_{\xi_1+\xi_2+\cdots+\xi_i}(s)\leqslant
\intl_0^{x-t} 1-\Phi(x-s)\ud F_{\xi_1+\xi_2+\cdots+\xi_i}(s)=\ldots
$$
and, by  integration  by  parts, we have
\begin{multline*}
I_i\leqslant (1-\Phi(x-s))F_{\xi_1+\xi_2+\cdots+\xi_i}(s)\bigg|_0^{x-t} - \intl_0^{x-t} F_{\xi_1+\xi_2+\cdots+\xi_i}(s) \ud  1-\Phi(x-s)=
\\
(1-\Phi(t))F_{\xi_1+\xi_2+\cdots+\xi_i}(t) + \intl_0^{x-t} F_{\xi_1+\xi_2+\cdots+\xi_i}(s) \ud  \Phi(x-s)\leqslant
\end{multline*}
now, using \eqref{sravn},
$$I_i\leqslant(1-\Phi(t))G^{n\ast}(s)+ \intl_0^{x-t} G^{n\ast}(s) \ud  (\Phi(x-s)-1) $$
anew, by integration  by parts, we have
$$
I_i\leqslant \intl_0^{x-t}1- \Phi(x-s)\ud G^{n*}(s).
$$

Thus,
\begin{multline*}
\PP\{B_t > x\}= 1-F_1(t+x)+\sum_i^\infty \intl_0^{x-t} 1-F_{i+1}(x-s)\ud F_{\xi_1+\xi_2+\cdots+\xi_i}(s) \leqslant \\
1-\Phi(x)+\sum_i^\infty \intl_0^{x-t} 1-\Phi(x-s)\ud G^{n\ast}(s)=\\
1-\Phi(x) +\intl_0^{x-t} 1-\Phi(x-s)\ud  H_G(s),
\end{multline*}
where $H_G(s)$ is a renewal function of the ``classic'' renewal process with the distribution of the renewal times $G(s)$; $H_G(s)\bd \suml_{n=1}^\infty G^{\ast n}(s)$.
Then,
$$
\EE\,  B_t\leqslant \left(\intl\limits_0^\infty ( 1-\Phi(x))+ \intl_0^{x-t} 1-\Phi(x-s)\ud H_G(s)\right) \ud x.
$$
So,
\begin{multline*}
\EE\,  B_t\leqslant \intl\limits_0^\infty  \left( 1-\Phi(x)\right) dx+\intl\limits_0^\infty \intl_0^{x-t} 1-\Phi(x-s)\ud H_G(s) \ud x=
\\
\EE\, \eta+\iint\limits_{\left\{\stackrel{0\leqslant x<\infty,}{0\leqslant s<x-t}  \right\}} (1-\Phi(x-s))H'_G(s) \ud x \ud s.
\end{multline*}
Now, we apply Smith's Key Renewal Theorem (see \cite{Smith}):
\begin{multline*}
\intl_0^\infty \intl_0^{x-t} 1- \Phi(x-s)\ud  H_G(s) (\ud t)= \\ \\
\intl_0^\infty \ud  H_G(s)  \intl_0^{x-t} 1- \Phi(x-s) (\ud t)=
\intl_0^\infty (x-s) [1- \Phi(x-s)] \ud H_G(s)= \\ \\
\mbox{ apply Smith's Theorem }\\ \\
=\frac{1}{\EE\,  \zeta} \intl_0^\infty (x+\theta)[1- \Phi(x+\theta)] \ud \theta.
\end{multline*}
Now,
$$
 \frac{1}{\EE\,  \zeta}\intl_0^\infty (x+\theta)[1- \Phi(x+\theta)] \ud (\theta +x)=\frac{1}{\EE\,  \zeta}
 \frac{1}{2} \intl_0^\infty 1- \Phi (v) \ud v^2 = \\
\frac{1}{\EE\,  \zeta}\times\frac{\EE\, (\eta^2)}{2}.
$$
The inequality (\ref{osn}) is proved.

The analogous calculations for the equation
\begin{multline*}\label{nedo}
\PP\{W_t>x\} =
\\
=\PP\{\xi_1>t+x\}+\sum_{i=1}^\infty \PP\big\{\{\xi_{i+1}>(t-s+x)\}\& \{\xi_1+\xi_2+\ldots+\xi_i\leqslant s\} \& \{s\leqslant t\}\big\} \leqslant
\\
\leqslant\PP\{\xi_1>x\}+\sum_{i=1}^\infty \PP\big\{\{\xi_{i+1}>(t-s)\}\& \{\xi_1+\xi_2+\ldots+\xi_i\leqslant s\} \& \{s\leqslant t\}\big\}
\end{multline*}
prove the inequality (\ref{osn1}).

The Theorem \ref{thm1} is proved.
\end{proof}


\section{Conclusion}
This fact is very important because \eqref{lord} is a uniform bound for any {\it fixed} (non-random) time $t$.
It was used for construction of strong bounds for some queueing systems end reliability systems.
For this aim, it can consider for some stochastic regenerative process described the behaviour of technical system an embedded renewal process, and to study the convergence rate of extended renewal Markov process. But in many practical situations, the counting process is not strongly renewal in a classic means.

\section{Acknowledgement}
The authors are grateful to Prof.L.G.Afanasyeva, Prof.S.A.Pirogov and  Prof.A.D.Manita for valuable advices and discussions.

\end{document}